\documentclass[10 pt, a4 paper]{amsart}
\usepackage{amssymb, amsmath, amsthm, amsfonts, amscd, enumerate, verbatim, calc}

\newtheorem{theorem}{Theorem}[section]
\newtheorem{lemma}[theorem]{Lemma}

\theoremstyle{definition}
\newtheorem{definition}[theorem]{Definition}
\newtheorem{example}[theorem]{Example}
\newtheorem{proposition}[theorem]{Proposition}

\theoremstyle{remark}
\newtheorem{remark}[theorem]{Remark}
\newtheorem{corollary}[theorem]{Corollary}
\newtheorem{question}[theorem]{Question}

\numberwithin{equation}{section}

\begin{document}

\title[Spectral representation of unbounded  normal $\mathcal{AM}$- operators]{Spectral representation of  absolutely minimum attaining  unbounded  normal operators}

\author{S. H. Kulkarni}
\address{Department of Mathematics, \\
IIT Palakkad, Ahalia Integrated Campus,\\
 Kozhippara, Palakkad, Kerala, India 678557.}
\email{shk@iitpkd.ac.in}
\thanks{}

\author{G. Ramesh}
\address{Department of Mathematics,\\ IIT Hyderabad, \\Telangana, India 502 284.}
\email{rameshg@math.iith.ac.in}
\thanks{}

\subjclass{Primary 47A75,47A05 ; Secondary 47A10,47A15 }

\date{\today}


\keywords{closed operator, minimum attaining, absolutely minimum attaining operator, invariant subspace, bounded transform, essential spectrum.}

\begin{abstract}
Let $T:D(T)\rightarrow H_2$ be a densely defined closed operator with domain $D(T)\subset H_1$. We say $T$ to be absolutely minimum attaining if for every closed subspace $M$ of $H_1$, the restriction operator $T|_M:D(T)\cap M\rightarrow H_2$ attains its minimum modulus $m(T|_{M})$. That is, there exists $x \in D(T)\cap M$ with $\|x\|= 1$ and $\|T(x)\| = \inf \{\|T(m)\|: m \in D(T) \cap M: \|m\|=1\}$. In this article, we prove several characterizations of this class of operators and show that every operator  in this class  has a nontrivial hyperinvariant subspace.

We also prove a spectral theorem for  unbounded normal operators of this class. It turns out that every such operator has a compact resolvent.
\end{abstract}

\maketitle

\section{Introduction }

The class of absolutely minimum attaining unbounded operators was introduced in \cite{SHKGRAM1} where its basic properties and structure were described under some additional assumptions. On the other hand,
similar studies in the case of bounded absolutely minimum operators were carried out in \cite{GRS2,NBGR2}. There is a significant difference between the bounded absolutely minimum attaining operators and the unbounded ones. The results in the present article improve results from \cite{SHKGRAM1}. It is interesting to note that this class contains densely defined closed operators with finite-dimensional null space and a compact generalized inverse. The complete characterization, structure, spectral properties and hyperinvariant subspaces are studied in \cite{GRS2,NBGR2,NBGR4}. A class larger than the absolutely norm attaining is explored in \cite{GROH1,GROH2}. Hence in this article we exclusively study the class of absolutely minimum attaining operators in the unbounded setting. Since the methods of the bounded case do not work in this case, we adopt different methods for proving our results.

The class of minimum attaining unbounded operators has been recently studied in \cite{SHKGRAM1,shkgrminattaining}. It is proved that this class is dense in the class of densely defined closed operators with respect to the gap metric. This result can be compared with the Lindenstrauss theorem of norm attaining operators. Moreover, a quantitative version of the Lindenstrauss theorem for minimum attaining operators is discussed in \cite{NBGR3}.

In this article we prove that a densely defined closed operator is absolutely minimum attaining if and only if it has a finite-dimensional null space and the Moore-Penrose inverse (generalized inverse) of the operator is compact. Using this characterization we study some properties of this class of operators. In particular, we prove a spectral theorem for  absolutely minimum attaining unbounded normal operators. The Spectral Theorem for unbounded normal operators is already available in the literature,
for example, in  \cite{rud}. Since the class considered by us is a subclass of this class, we need to emphasize that the spectral theorem proved by us is of a different type
and should be compared with the spectral theorem for compact normal operators.  In the end, we show that every such  operator has a non-trivial hyperinvariant subspace.

 In the second section we provide basic results which will be used  throughout the article. In the third section, we prove main results of this paper.

\section{Preliminaries}
In this section we give details of  basic notations, definitions and results which we need to prove our main results.

Throughout we work  with infinite-dimensional complex Hilbert spaces denoted by $H, H_1,H_2$ etc. The inner product and the
induced norm on these spaces are denoted  by  $\langle \cdot\rangle$ and
$\|.\|$, respectively. If $T$ is a linear operator with domain $D(T)$, a subspace of $H_1$ and taking values in $H_2$, then $T$ is said to be densely defined if $D(T)$ is dense in $H_1$.  The  graph $G(T)$ of $T$ is defined by  $G(T):={\{(x,Tx):x\in D(T)}\}\subseteq H_1\oplus H_2$. We say $T$ to be a closed operator if $G(T)$ is closed.  Equivalently, we can say that  $T$ is closed if and only if  for every sequence $(x_n)$ in  $D(T)$  such that  $x_n\rightarrow x\in H_1$ and $Tx_n\rightarrow y\in H_2$, then $x\in D(T)$ and $Tx=y$.


The closed graph Theorem \cite{rud} assert that  an everywhere defined
closed operator is bounded.  That is,  the domain of an unbounded closed operator must be a proper subspace of a Hilbert space.

The space of all bounded operators between $H_1$ and $H_2$ is denoted by $\mathcal B(H_1,H_2)$ and the class of all closed operators between $H_1$ and $H_2$ is denoted by $\mathcal C(H_1,H_2)$. We write $\mathcal B(H)$ for $\mathcal B(H,H)$ and $\mathcal C(H)$ for $\mathcal C(H,H)$. We denote the space of all compact operators between $H_1$ and $H_2$ by $\mathcal K(H_1,H_2)$ and $\mathcal K(H,H)$ by $\mathcal K(H)$. Let us denote by $\mathcal F(H_1,H_2)$, the space of all finite rank bounded operators from $H_1$ into $H_2$ and by $\mathcal F(H)$, the space of all finite rank bounded operators on $H$.

The unit sphere of a subspace $M$ of $H$ is defined by  $S_M:={\{x\in M:\|x\|=1}\}$. If $M$ is closed, then the othogonal projection of $H$ onto $M$ is denoted by $P_{M}$.

The null space and the range space of $T\in \mathcal C(H_1,H_2)$ are denoted by $N(T)$ and $R(T)$ respectively and the space $C(T):=D(T)\cap N(T)^\bot$ is called the carrier of $T$. In fact, $D(T)=N(T)\oplus^\bot C(T)$ \cite[page 340]{ben}.

If $S$ and $T$ are closed operators  with the property that $D(T)\subseteq D(S)$ and $Tx=Sx$ for all $x\in D(T)$, then $T$ is called the restriction of $S$ denoted by $T\subseteq S$, and $S$ is called an extension of $T$, which is denoted by $T\subseteq S$.  Furthermore, $S=T$ if and only if  $S\subseteq T$ and $T\subseteq S$.

If $T$ is a densely defined operator, then there exists a unique linear operator (in fact, a closed operator) $T^*:D(T^*)\rightarrow H_1$, with
\begin{equation*}
D(T^*):={\{y\in H_2: x\rightarrow \langle Tx,y\rangle \, \text{for all}\, x\in D(T)\,\text{is continuous}}\}\subseteq H_2
\end{equation*}
 satisfying $\langle Tx,y\rangle =\langle x,T^*y\rangle$ for all $x\in D(T)$ and $y\in D(T^*)$.

A densely defined operator $T\in \mathcal C(H)$ is said to be
\begin{enumerate}
\item normal if $T^*T=TT^*$
\item  self-adjoint if $T=T^*$
\item  positive if $\langle Tx,x\rangle \geq 0$ for all $x\in D(T)$
\item symmetric if $T\subseteq T^*$. In other words, $T$ is symmetric if $\langle Tx,y\rangle=\langle x,Ty\rangle$ for every $x,y\in D(T)$.
\end{enumerate}

If $T$ is positive, then there exists a unique positive operator $S$ such that $T=S^2$. The operator $S$ is called the square root of $T$ and it is denoted by $S=T^\frac{1}{2}$.

A bounded operator $V\in \mathcal B(H_1,H_2)$ is said to be a partial isometry if $\|Vx\|_2=\|x\|_1$ for all $x\in N(V)^{\bot}$. In this case $N(V)^{\bot}$ is called the initial space of $V$ and $R(V)$ is called the final space of $V$.

 If $T\in \mathcal C(H_1,H_2)$ is densely defined, then the  operator  $|T|:=(T^*T)^\frac{1}{2}$ is called the modulus of $T$. There exists a unique partial isometry $V:H_1\rightarrow H_2$ with initial space $\overline{R(T^*)}$ and the final space $\overline{R(T)}$ such that $T=V|T|$. This factorization of $T$ is called the polar factorization or the polar decomposition of $T$.

 It can be verified that $D(|T|)=D(T)$ and $N(|T|)=N(T)$ and $\overline{R(|T|)}=\overline{R(T^*)}$.


 Let $T\in \mathcal C(H)$ be densely defined. The resolvent of $T$ is defined by
 \begin{equation*}
 \rho(T):={\{\lambda \in \mathbb C:  (T-\lambda I)^{-1}\in \mathcal B(H)}\}.
 \end{equation*}
The set $\sigma(T)=\mathbb C \setminus \rho(T)$ is called the spectrum of $T$. The set $$\sigma_p(T)={\{\lambda \in \mathbb C: T-\lambda I:D(T)\rightarrow H \; \text{is not one-to-one}}\}$$ is called the point spectrum of $T$.



For a self-adjoint operator $T\in \mathcal C(H)$, the \textit{discrete spectrum} $\sigma_d(T)$ is defined as  the set of all isolated eigenvalues of $T$ with finite multiplicity. The set $\sigma_{ess}(T)=\sigma(T)\setminus \sigma_d(T)$ is called the \textit{essential spectrum} of $T$. For more details of this concept we refer \cite[Definition 8.3, page 178]{schmudgen}.

We refer \cite{taylorlay,goldberg,rud,birmannsolomyak,schmudgen} for the above basics of unbounded operators.

Here we recall the definition and properties of the Moore-Penrose inverse (or generalized inverse) of a densely defined closed operator that we need for our purpose.

If $T\in \mathcal C(H_1,H_2)$ is densely defined, then there exists
a unique densely defined operator $T^\dagger \in \mathcal
C(H_2,H_1)$ with domain $D(T^\dagger)=R(T)\oplus ^\bot R(T)^\bot$
and has the following properties:
\begin{enumerate}
\item $TT^\dagger y=P_{\overline{R(T)}}~y, ~\text{for all}~y\in D(T^\dagger)$

\item $T^\dagger Tx=P_{N(T)^\bot} ~x, ~\text{for all}~x\in D(T)$

\item $N(T^\dagger)=R(T)^\bot$.
\end{enumerate}
The  operator $T^\dagger$ is called the \textit{Moore-Penrose inverse} of $T$. An alternative  definition of $T^\dagger$ is given below.

For every $y\in D(T^\dagger)$, let $$L(y):=\Big\{x\in D(T):
\|Tx-y\|\leq \|Tu-y\|\quad \text{for all} \quad u\in D(T)\Big\}.$$
 Here any $u\in L(y)$ is called a \textit{least square solution} of the operator equation $Tx=y$. The vector  $x=T^\dagger y\in L(y),\, \text{and}\;\|T^\dagger y\|\leq \|x\|\quad \text{for all} \quad u\in L(y)$
 and it is called the  \textit{least square solution of minimal norm}.
 A different treatment of $T^\dagger$ is given in \cite[Pages 314, 318-320]{ben},
 where the authors call this as ``\textit{the Maximal Tseng generalized Inverse}".

 Next we define minimum attaining operators and the absolutely minimum attaining operators.

 \begin{definition}\cite{ben, goldberg}\label{minmmodulus}
 Let $T\in \mathcal C(H_1,H_2)$ be densely defined. Then
 \begin{equation*}
 m(T):=\inf{\{\|Tx\|: x\in S_{D(T)}}\}
 \end{equation*}
 is  called the minimum modulus of $T$. The operator $T$ is said to be bounded  below if and only if $m(T)>0$.
\end{definition}

\begin{remark}

If $T\in \mathcal C(H_1,H_2)$ is densely defined, then
\begin{itemize}
\item[(i)] $m(T)>0$ if and only if $R(T)$ is closed and $T$ is one-to-one
\item[(ii)] Since $D(T)=D(|T|)$ and $\|Tx\|=\|T|x\|$ for all $x\in D(T)$, we can conclude that $m(T)=m(|T|)$.
    \end{itemize}
\end{remark}
 \begin{definition}
   Let $T\in \mathcal C(H_1,H_2)$ be densely defined. Then $T$ is said to be
\begin{enumerate}
\item \textit{minimum attaining} if there exists $x_0\in S_{D(T)}$ such that $\|Tx_0\|=m(T)=\inf{\{\|Tx\|:x\in S_{D(T)}}\}$
\item \textit{absolutely minimum attaining} if for every closed subspace $M$ of $H_1$, the operator $T|_{M}:M\cap D(T)\rightarrow H_2$ is minimum attaining.
\end{enumerate}
 \end{definition}
We denote the set of all densely defined closed absolutely minimum attaining operators between $H_1$ and $H_2$ by $\mathcal{AM}_{c}(H_1,H_2)$ and the set of all densely defined minimum attaining closed operators by $\mathcal M_{c}(H_1,H_2)$. We write $\mathcal{AM}_{c}(H,H)$ and $\mathcal M_{c}(H,H)$ by $\mathcal{AM}_{c}(H)$ and $\mathcal M_{c}(H)$, respectively.

In particular, if $T\in \mathcal B(H_1,H_2)$, then $T$ is called minimum attaining if there exists $x\in H_1$ with $\|x\|=1$ such that $\|Tx\|=m(T)$. We say $T$  to be absolutely minimum attaining if for every closed subspace $M$ of $H_1$, the restriction operator $T|_{M}:M\rightarrow H_2$ is minimum attaining. This class is denoted by $\mathcal{AM}(H_1,H_2)$. If $H_1=H_2=H$, then we denote $\mathcal{AM}(H_1,H_2)$ by $\mathcal{AM}(H)$.

 Let $M$ be a closed subspace of $H$ and $T\in \mathcal C(H)$ be densely defined. Then $M$ is said to be invariant under $T$,
if $T(M\cap D(T))\subseteq M$.

Let $P:=P_M$. If $P(D(T))\subseteq D(T)$ and $(I-P)(D(T))\subseteq D(T)$, then
\begin{equation*}
T= \left(
           \begin{array}{cc}
             T_{11} & T_{12} \\
             T_{21} & T_{22} \\
           \end{array}
         \right),
\end{equation*}
where $T_{ij}=P_iTP_j|M_j\; (i,j=1,2)$, $M_1=P(D(T))$ and $M_2=(I-P)(D(T))$. Here $P_1=P$ and $P_2=I-P$. It is known that $M$ is invariant under $T$ if and only if $T_{21}=0$. Also, $M$ reduces $T$ if and only if $T_{21}=0$ and $T_{12}=0$.

\section{Main Results}
In this section we prove our main results. First we recall a few basic results on absolutely minimum attaining and absolutely norm attaining operators, that we use frequently. Recall that $T\in \mathcal B(H_1,H_2)$ is norm attaining if there exists $x\in S_{H_1}$ such that $\|Tx\|=\|T\|$. We say $T$ to be absolutely norm attaining if for every closed subspace $M$ of $H_1$, the restriction operator $T|_{M}:M\rightarrow H_2$ is norm attaining. We denote the set of all absolutely norm attaining operators by $\mathcal{AN}(H_1,H_2)$. For more details of this class operators we refer to \cite{RAMpara}.
\begin{theorem}\cite[Theorem 5.14]{GRS2}
Let $T\in \mathcal B(H_1,H_2)$. Then the following are equivalent;
\begin{enumerate}
\item $T\in \mathcal{AM}(H_1,H_2)$
\item $T^*T\in \mathcal{AM}(H_1)$.
\end{enumerate}
\end{theorem}

\begin{theorem}\cite[Theorem 4.6]{SHKGRAM1}\label{inverseAN}
Let $T\in \mathcal C(H)$ be densely defined and $T^{-1}\in \mathcal B(H)$. Then $T\in \mathcal{AM}_{c}(H)$ if and only if $T^{-1}\in \mathcal{AN}(H)$.
\end{theorem}

Let $T\in \mathcal C(H)$ be a densely defined closed operator. Then $Z_{T}=T(I+T^*T)^{-\frac{1}{2}}$ is called the \textit{bounded transform} of $T$. Here we list a few important properties of $Z_{T}$.
\begin{theorem}\cite[Page 90]{schmudgen}
  Let $T\in \mathcal C(H)$ be densely defined. Then we have
  \begin{enumerate}
    \item $\|Z_T\|\leq 1$.     \item $(Z_{T})^*=Z_{T^*}$
    \item $Z_{T}^*Z_{T}=I-(I+T^*T)^{-1}$.
  \end{enumerate}
\end{theorem}
It is to be noted that in particular, $T\in \mathcal B(H)$ if and only $\|Z_{T}\|<1$ (see \cite[Corollary 2.1]{Sharifigap}). Also we can easily prove that $N(Z_T)=N(T)$ and $R(T)=R(Z_T)$ (see \cite{rameshthesis}).


\begin{theorem}\cite[theorem VIII.3]{reedsimon1}
  Let $S$ be a densely defined closed symmetric operator. Then
  \begin{equation}\label{symopproperty}
    \|(S+iI)x\|^2=\|Sx\|^2+\|x\|^2,\; \text{for all}\; x\in D(S).
  \end{equation}
\end{theorem}
\begin{theorem}\label{complexperturb}
Let $S$ be a densely defined closed symmetric operator. Then the following are true;
\begin{enumerate}
  \item \label{minmodformula} $m(S+iI)=\sqrt{1+m(S)^2}$
  \item \label{minattainingreln} $S\in \mathcal M_{c}(H)$ if and only if $S+iI\in \mathcal M_{c}(H)$
    \item \label{AMrelation}$S\in \mathcal {AM}_{c}(H)$ if and only if $S+iI\in \mathcal {AM}_{c}(H)$.
\end{enumerate}
\end{theorem}
\begin{proof}
All the proofs  directly follow by Equation \ref{symopproperty} and the definitions of the minimum modulus, minimum attaining property and the absolutely minimum attaining property, respectively.
  \end{proof}

\begin{theorem}\label{equivalentAMs}
  Let $T$ be a positive  and unbounded operator. Then the following are equivalent;
  \begin{enumerate}
    \item\label{positiveAM} $T\in \mathcal{AM}_{c}(H)$
    \item\label{AMplusidentity} $T^2+I\in \mathcal{AM}_{c}(H)$
    \item \label{boundedtransformAM} $Z_{T}\in \mathcal{AM}(H)$.
  \end{enumerate}
\end{theorem}
\begin{proof}
Proof of (\ref{positiveAM}) if and only (\ref{AMplusidentity}): By Theorem \ref{complexperturb}, we know that $T\in \mathcal{AM}_{c}(H)$ if and only $T+iI\in \mathcal{AM}_{c}(H) $. Since $T+iI$ is one to one, by applying \cite[Theorem 4.10]{SHKGRAM1}, we can conclude that $T+iI\in \mathcal{AM}_{c}(H) $ if and only if $T^2+I=(T+iI)^*(T+iI)\in \mathcal{AM}_{c}(H)$.

To prove the equivalence of (\ref{AMplusidentity}) and (\ref{boundedtransformAM}), first we observe that
\begin{equation}\label{bddtransformeqn}
  Z_{T}^*Z_{T} =T^2(I+T^2)^{-1}=I-(I+T^2)^{-1}.
\end{equation}
  Now, if $I+T^2\in \mathcal{AM}_{c}(H)$, then $(I+T^2)^{-1}$ is absolutely norm attaining, by Theorem \ref{inverseAN}. Hence by \cite[Theorem 2.5]{VNGR1}, there exists a compact positive operator $K$ and positive finite rank operator $F$ and $\alpha \geq 0$ such that
  \begin{equation}\label{interimeq2}
   (I+T^2)^{-1}=\alpha I+K-F
  \end{equation}
 with $KF=0$ and $F\leq \alpha I$. Now, by Equation \ref{bddtransformeqn}, we have that $Z_{T}^{*}Z_{T}=(1-\alpha )I-K+F$. Our idea is to apply  \cite[Theorem 5.8]{GRS2} and  conclude  $Z_{T}^*Z_{T}\in \mathcal{AM}(H)$. Then this will imply $Z_{T}\in \mathcal{AM}(H)$. For this purpose, we need to prove that $\|K\|\leq 1-\alpha$. By post multiplying Equation \ref{interimeq2} and using the facts that $KF=0$ and $K$ commute with $(I+T^2)^{-1}$, we get

 \begin{equation}\label{interimeq3}
  (I+T^2)^{-1}K=(\alpha I+K)K.
 \end{equation}
Since $\|(I+T^2)^{-1}\|\leq 1$, we have $(I+T^2)^{-1}K\leq K$ or $(\alpha I+K)K\leq K$. This implies that $K^2\leq (1-\alpha)K\leq (1-\alpha)\|K\|$. From this we can conclude that $\|K^2\|=\|K\|^2\leq (1-\alpha)\|K\|$ or $\|K\|\leq 1-\alpha$.

  To prove the implication $(\ref{boundedtransformAM})\Rightarrow (\ref{AMplusidentity})$, let $Z_{T}\in \mathcal{AM}(H)$. Then $Z_{T}^*Z_{T}\in \mathcal{AM}(H)$. Hence  there exists $K\in \mathcal K(H)_{+},\; F\in \mathcal F(H)_{+}$ and $\beta \geq 0$ satisfying $KF=0$ and $K\leq \beta I$ such that $Z_{T}^*Z_{T}=\beta I-K+F$. Hence by Equation \ref{bddtransformeqn}, we have that
  \begin{equation*}
   (T^2+I)^{-1}=(1-\beta)I+K-F\in \mathcal{AN}(H),
  \end{equation*}
  by \cite[Theorem 5.2]{SP}.
    Now by Theorem \ref{inverseAN}, we can conclude that $I+T^2\in \mathcal{AM}_{c}(H)$.
\end{proof}

Our next goal before proving the spectral theorem is to show that  an absolutely minimum attaining closed operator can either have a finite-dimensional null space or finite-dimensional range space. This property is helpful in deciding the spectrum of such an operator. To achieve this we prove the following results.

\begin{lemma}\label{normattainingbddtransform}
  Let $T\in \mathcal C(H)$ be densely defined. If $Z_{T}$ is norm attaining, then $T\in \mathcal B(H)$.
\end{lemma}

\begin{proof}
We know that $\|Z_{T}\|\leq 1$. If $\|Z_T\|<1$, then clearly $T\in \mathcal B(H)$, by \cite[Corollary 2.1]{Tharifigap}. Next assume that $\|Z_{T}\|=1$. Since $Z_T$ is norm attaining, there exists $x_0\in S_H$ such that $Z_T^*Z_Tx_0=x_0$. That is, $(I-(I+T^*T)^{-1})x_0=x_0$ or $(I+T^*T)^{-1}{x_0}=0$. This imply that $x_0=0$, a contradiction. Hence, the assumption that $\|Z_T\|=1$ is wrong. That is, $\|Z_T\|<1$.  Now the conclusion follows by the earlier case.
\end{proof}

\begin{proposition}
 Let $T\in \mathcal C(H_1,H_2)$ be densely defined. If $T$ has finite rank, then $T\in \mathcal B(H_1,H_2)$.
\end{proposition}
\begin{proof}
 Since $T$ is closed, $N(T)$ is  closed and hence $D(T)=N(T)\oplus C(T)$, where $C(T)=D(T)\cap N(T)^{\bot}$. Note that $T_0=T|_{C(T)}$ is a bijection from $C(T)$ onto $R(T)$. This implies that $C(T)$ is finite-dimensional and hence $T_0$ is bounded. Suppose its norm is $M$ and consider any $x$ in the domain $D(T)$ of $T$. Then $x$ can be written uniquely as
$x = u + v$ with $u \in N(T)$ and $v\in C(T)$. Since $D(T)$ is a subspace, $v = x - u$ is in $D(T)$. Now
\begin{equation*}
\|T(x)\| = \|T(v)\| \leq M \|v\| \leq M \|x\|.
\end{equation*}
This shows that $T$ is bounded on $D(T)$. Since $D(T)$ is an orthogonal
direct sum of a closed subspace $N(T)$ and a finite-dimensional subspace $C(T)$,  it is closed. As $T$ is densely defined, we obtain that $D(T)=H$.
\end{proof}

\begin{corollary}\label{infiniterange}
 Let $T\in \mathcal C(H_1,H_2)$ be densely defined and unbounded. Then $R(T)$ must be infinite-dimensional.
\end{corollary}

To prove the converse of Lemma \ref{normattainingbddtransform}, we need the following result.
\begin{theorem}\label{finiterankchar}
  Let $T$ be an unbounded positive operator $T\in \mathcal{AM}_{c}(H)$. Then  $N(T)$ is finite-dimensional.
\end{theorem}
\begin{proof}
 It is clear that $R(T)$ is infinite  dimensional by  Corollary \ref{infiniterange}. Suppose that $N(T)$ is infinite-dimensional. Then $N(Z_{T})$ is infinite-dimensional and hence $Z_{T}\in \mathcal F(H)$ by \cite[Remark 3.2]{NBGR2}. Since $R(T)=R(Z_{T})$, $T$ is a finite rank operator. Since $Z_T$ is norm attaining, by Lemma \ref{normattainingbddtransform}, we can conclude that $T$ is bounded.  This completes the proof.
\end{proof}

\begin{corollary}\label{finitedimnullspace}
  Let $T\in \mathcal {AM}_{c}(H_1,H_2)$ be densely defined and unbounded. Then $N(T)$ is finite-dimensional.
\end{corollary}
\begin{proof}
Since $T\in \mathcal{AM}_{c}(H_1,H_2)$, we have $|T|\in \mathcal{AM}_{c}(H_1)$ and by Theorem \ref{finiterankchar}, we can conclude that $N(|T|)=N(T)$ is finite dimensional.
\end{proof}

\begin{remark}
 If $A\in \mathcal{AM}(H_1,H_2)$, then either $R(A)$ or $N(A)$ is finite-dimensional (see \cite[Proposition 5.19]{GRS2} for details). But  Corollary \ref{finitedimnullspace} says that it is not the case  when the operator is unbounded. The reason for this difference is that the finite rank operators are included in the class of bounded absolutely minimum attaining operators.
\end{remark}

Next, we describe the structure of unbounded positive absolutely minimum attaining operators. This result generalizes that of \cite[Theorem 4.8]{SHKGRAM1}.
\begin{theorem}\label{posAMstructure}
Let $T\in \mathcal{AM}_{c}(H)$ be positive and unbounded. There exists an unbounded (increasing) sequence ${\{\lambda_n}\}$ of eigenvalues of $T$ with corresponding orthonormal eigenvectors ${\{\phi_n}\}$ such that
\begin{enumerate}
 \item \label{domain}$D(T)=\left \{  x\in H: \displaystyle \sum_{n=1}^{\infty} \lambda_n^2|\langle x,\phi_n\rangle |^2<\infty \right \} $ and
 $$Tx=\displaystyle \sum_{n=1}^{\infty} \lambda_n \langle x,\phi_n\rangle \phi_n,\; \text{for all}\; x\in D(T)$$
 \item \label{spectrum} $\sigma(T)\subseteq {\{\lambda_n:n\in \mathbb N}\}\cup {\{0}\}=\sigma_{p}(T)\cup {\{0}\}$
 \item \label{discreteness} If $\mu \in \sigma_{p}(T)$, then it is an eigenvalue with finite multiplicity
 \item \label{eigenbasis}$\overline{\text{span}}{\{\phi_n:n\in \mathbb N}\}=N(T)^{\bot}$
  \item \label{diagonalizability}$T$ is diagonalizable. That is, there exists an orthonormal basis of  $H$ consisting of  eigenvectors of $T$
   \item\label{minmattainingonsubset}    For every subset $S$ of $\mathbb N$, we have  \begin{equation*}
     \inf{\{\lambda_n:n\in S}\}=\min{\{\lambda_n:n\in S}\}.
   \end{equation*}

\end{enumerate}

\begin{proof}
If $T$ is one-to-one then all the statements (\ref{domain}) to (\ref{eigenbasis}) follow from \cite[Theorem 4.8]{SHKGRAM1}. If $T$ is not one-to-one, then by Corollary \ref{finitedimnullspace}, $N(T)$ is finite-dimensional. That is, $0\in \sigma_p(T)$. Since $T=0\oplus T|_{N(T)^{\bot}}$,  applying \cite[Theorem 4.8]{SHKGRAM1} to $T|_{N(T)^{\bot}}$, the conclusion follows.

The proof of (\ref{diagonalizability}) follows by taking an orthonormal basis of $N(T)$ and adjoining it to the eigenbasis of $N(T)^{\bot}$ obtained in (\ref{eigenbasis}). To prove (\ref{minmattainingonsubset}), first
note that by (\ref{diagonalizability}), $T$ is diagonalizable. First assume that
$N(T)\neq  {\{0}\}$. Then $N(T)$ is finite-dimensional and hence $0 \in  \sigma_d(T)$.
In this case $\sigma(T) = \sigma_d(T) = {\{\lambda_n : n \in \mathbb  N}\} \cup {\{0}\}$.
We write $T = 0 \oplus T_0$, where $T_0 = T|_{N(T)^\bot}$. Since
$T_0 \in \mathcal{AM}(N(T)^{\bot})$ and $R(T_0) = R(T)$ is closed as $T \in \mathcal{ AM}_c(H)$, we conclude that $T_0$ has a bounded inverse and by Theorem \ref{inverseAN}, we
have $T^{-1}_0\in \mathcal{AN}(N(T)^{\bot})$.  Hence $\sigma(T_{0}^{-1}) = {\{\lambda_n^{-1}:n\in \mathbb{N} }\}$. Now by \cite[ Theorem 3.8(i)]{SP}, we have $\max {\{\lambda_k : k \in S \setminus {0}}\} =
\sup {\{\lambda_k^{-1} : k \in S \setminus  {0}}\}$. From this we get that $\inf {\{\lambda_k : k \in S \setminus {0}}\}=
\min {\{\lambda_k : k \in S \setminus {0}}\}$. Note that if $0 \in  S$, then
\begin{equation*}
\inf {\{\lambda_k : k \in S}\} = 0 =\min {\{\lambda_k : k \in  S}\}.
\end{equation*}
This completes the proof.
  \end{proof}
\end{theorem}

\begin{corollary}\label{noinfinitedmnlssp}
  Let $T\in \mathcal C(H_1,H_2)$ be densely defined and unbounded. Then  $T\in \mathcal{AM}_{c}(H_1,H_2)$ if and only if $D(T)$ does not contain an infinite-dimensional closed subspace.
\end{corollary}
\begin{proof}
  We have that $T\in \mathcal{AM}_{c}(H_1,H_2)$ if and only if  $|T|\in \mathcal{AM}_{c}(H_1)$. By Theorem \ref{posAMstructure}, this is equivalent to the fact that  $|T|$ has pure discrete spectrum. By \cite[Page 113, Exercise 25]{schmudgen}, this is equivalent to the fact that  $D(|T|)=D(T)$ does not contain an infinite-dimensional closed subspace.
\end{proof}

\begin{proposition}\label{squareproperty}
  Let $T\in \mathcal C(H)$ be positive and unbounded. If $T\in \mathcal{AM}_{c}(H)$, then $T^2\in \mathcal{AM}_{c}(H)$.
\end{proposition}
\begin{proof}
 Note that if $T$ is one-to-one, then the result holds true by \cite[Theorem 4.10]{SHKGRAM1}. So let us assume that $N(T)\neq{\{0}\}$. Then $T=0\oplus T_0$, where $T_0=T|_{N(T)^{\bot}}$. By Corollary \ref{finitedimnullspace}, $N(T)$ is finite-dimensional. Clearly, by definition, $T_0\in \mathcal{AM}_{c}(N(T)^{\bot})$ and consequently by \cite[Theorem 4.10]{SHKGRAM1}, $T_0^2\in \mathcal{AM}_{c}(N(T)^{\bot})$. Hence by Theorem \ref{equivalentAMs}, $Z_{T_0^2}\in \mathcal{AM}(N(T)^{\bot})$. Since $N(Z_{T^2})=N(T^2)=N(T)$ is finite-dimensional, by \cite[Proposition 3.6]{NBGR2}, we can conclude that $Z_{T^2}=0\oplus Z_{T_0^2}\in \mathcal{AM}(H)$. Now the conclusion follows by Theorem \ref{equivalentAMs}.
  \end{proof}

  \begin{proposition}\label{scalarperturbAM}
  If $T\in \mathcal {AM}_{c}(H)$ is positive, unbounded and $\alpha>0$, then $T+\alpha I \in \mathcal{AM}_{c}(H)$.
\end{proposition}
\begin{proof}
First note that since $T\in \mathcal{AM}_{c}(H)$, it follows that $R(T)$ is closed. Also, $N(T)$ is finite-dimensional by Theorem \ref{finiterankchar}. Since $N(T)$ is a reducing subspace for $T$, we have that
$T=\left(
   \begin{array}{cc}
     0 & 0 \\
     0 & T_0 \\
   \end{array}
 \right)
$, where $T_0=T|_{N(T)^{\bot}}$. Since $T_0$ is bijective and  $T_0\in \mathcal{AM}_{c}(N(T)^{\bot})$, it follows that $T_0+\alpha I_{N(T)^{\bot}}\in \mathcal{AM}_{c}(N(T)^{\bot})$. Hence by Theorem \ref{inverseAN}, we can conclude that   $(T_0+\alpha I_{N(T)^{\bot}})^{-1}\in \mathcal{K}(N(T)^{\bot})$. Hence we have $(T+\alpha I)^{-1}=\left(
                                                              \begin{array}{cc}
                                                              \alpha I_{N(T)} & 0 \\
                                                                0 & (T_0+\alpha I_{N(T)^{\bot}})^{-1} \\
                                                              \end{array}
                                                            \right)\in \mathcal K(H)$.
Now the conclusion follows by Theorem \ref{inverseAN}.
\end{proof}

\begin{corollary}\label{sqrtproperty}
Let $T\in \mathcal C(H)$ be positive and unbounded. If $T^2\in \mathcal{AM}_{c}(H)$, then $T\in \mathcal{AM}_{c}(H)$.
  \end{corollary}
  \begin{proof}
  If $T^2\in \mathcal{AM}_{c}(H)$, then $T^2+I\in \mathcal{AM}_{c}(H)$ by Proposition \ref{scalarperturbAM}. But we have $T^2+I=(T-iI)(T+iI)=(T+iI)^*(T+iI)\in \mathcal{AM}_{c}(H)$. Since $T+iI$ is one-to-one, by \cite[Theorem 4.11]{SHKGRAM1}, we have $T+iI\in \mathcal{AM}_{c}(H)$. Next, by Theorem \ref{complexperturb}, we can conclude that $T\in \mathcal{AM}_{c}(H)$.
  \end{proof}

\begin{corollary}
  Let $T\in \mathcal C(H_1,H_2)$ be densely defined and unbounded. Then $T\in \mathcal{AM}_{c}(H_1,H_2)$ if and only if $T^*T\in \mathcal {AM}_{c}(H_1)$.
\end{corollary}
\begin{proof}
We know that $T\in \mathcal {AM}_{c}(H_1,H_2)$ if and only if $|T|\in \mathcal{AM}_{c}(H_1)$. By Corollary \ref{sqrtproperty}, this is equivalent to the fact $T^*T=|T|^2\in \mathcal{AM}_c(H_1)$. The converse follows by Proposition \ref{squareproperty}.
\end{proof}
%
%
%
%
%

\begin{theorem}
  Let $T\in \mathcal C(H)$ be positive. Then $T\in \mathcal K(H)$ if and only if $Z_{T}\in \mathcal K(H)$.
\end{theorem}
\begin{proof}
Assume that  $T\in \mathcal K(H)$. Since $(I+T^{2})^{-\frac{1}{2}}\in \mathcal B(H)$, we get that $Z_{T}=T(I+T^2)^{-\frac{1}{2}}\in \mathcal K(H)$. Next assume that $Z_{T}\in \mathcal K(H)$. That is $Z_{T}$ is norm attaining. Hence by Lemma \ref{normattainingbddtransform}, it is clear that $T\in \mathcal B(H)$. Hence $(I + T^2)^{1/2} \in B(H)$. Thus
$T = Z_T(I + T^2)^{1/2} \in \mathcal{K}(H)$.
\end{proof}

The following result is proved in \cite[Theorem 4.16]{SHKGRAM1}. Here we give a simple proof in a particular case.
\begin{proposition}\label{compactmpi}
 Let $T\in \mathcal C(H)$ be positive and unbounded. If $T\in \mathcal{AM}_{c}(H)$ then $N(T)$ is finite-dimensional and $T^{\dagger}\in \mathcal K(H)$.
\end{proposition}

\begin{proof}
Clearly $N(T)$ is finite-dimensional.  If $T$ is one-to-one, then the conclusion follows by \cite[Theorem 4.10]{SHKGRAM1}. Next, assume that $N(T)\neq {\{0}\}$. Then $T$ has the representation
$\left(
   \begin{array}{cc}
     0 & 0 \\
     0 & T_0 \\
   \end{array}
 \right)
$, where $T_0=T|_{N(T)^{\bot}}$. Clearly, $T_0\in \mathcal{AM}_{c}(N(T)^{\bot})$. Hence $T_0^{-{1}}\in \mathcal{K}(N(T)^{\bot})$, by Theorem \ref{inverseAN}. Hence $T^{\dagger}=0\oplus T_0^{-1}$ is compact.

\end{proof}

Next we illustrate the above result with an example.
\begin{example}\label{counterexmpcompact}
  Let $A$ be a linear operator with $$D(A)={\{(x_n)\in \ell^2:(nx_n)\in \ell^2}\}.$$ Define $A$ by
  \begin{equation*}
    A(x_1,x_2,x_3,\dots)=(x_1,2x_2,3x_3,\dots),\; (x_n)\in \ell^2.
  \end{equation*}
  It can be shown that $A$ is a densely defined closed operator. In fact, $A$ is positive and invertible with $A^{-1}:\ell^2\rightarrow \ell^2$ given by
  \begin{equation*}
    A^{-1}(y_1,y_2,y_3,\dots)=(y_1,\frac{y_2}{2},\frac{y_3}{3},\dots),\; (y_n)\in \ell^2.
  \end{equation*}
  Now define $T$ on $D:=\ell^2\oplus D(A)\subset \ell^2\oplus \ell^2$ by  $T=0\oplus A$. Then $T^{\dagger}=0\oplus A^{-1}$. It can be easily seen that $N(T)=\ell^2\oplus {\{0}\}$ which is infinite-dimensional. Also, $T^{\dagger}\in \mathcal K(H)$. We claim that $T\notin \mathcal{AM}_{c}(H)$. It is clear that $N(T)$ is a reducing subspace for $T$. It can be easily shown that $Z_{T}=0\oplus Z_{A}$. A routine computation shows that
  $Z_{A}$ is defined on the standard orthonormal basis ${\{e_n:n\in \mathbb N}\}$ of $\ell^2$ by
  \begin{equation*}
  Z_{A}(e_n)=\dfrac{n}{\sqrt{1+n^2}}e_n,n\in \mathbb N.
   \end{equation*}
   Note that $1$ is the limit point of the point spectrum $$\sigma_p(Z_{A})={\{\dfrac{n}{\sqrt{1+n^2}}:n\in \mathbb N}\}.$$ Hence $\sigma_{ess}(Z_{A})={\{1}\}$ and since $0$ is an eigenvalue of $Z_{T}$ with infinite multiplicity, we can conclude that $0\in \sigma_{ess}(T)$. So $\sigma_{ess}(Z_{T})={\{0,1}\}$. By \cite[Theorem 3.10]{NBGR2}, $Z_{T}\notin \mathcal{AM}(H)$. Now, by Theorem \ref{equivalentAMs}, we can conclude that $T\notin \mathcal{AM}_{c}(H)$.
\end{example}

\begin{proposition}\label{cptmpiAM}
  Let $T\in \mathcal {AM}_{c}(H_1,H_2)$ be densely defined and unbounded. Then $N(T)$ is finite-dimensional and $T^{\dagger}\in \mathcal K(H_2,H_1)$.
\end{proposition}
\begin{proof}
It is clear that $N(T)$ is finite-dimensional by  Corollary \ref{finitedimnullspace}. As $T\in \mathcal{AM}_{c}(H_1,H_2)$, we have $|T|\in \mathcal{AM}_{c}(H_1)$. This implies that $|T|^{\dagger}=|(T^*)^{\dagger}|\in \mathcal K(H_1)$. From this we can conclude that $T^{\dagger}\in \mathcal K(H_2,H_1)$.
\end{proof}

The following question is asked in \cite[Question 4.17]{SHKGRAM1}.
\begin{question}\label{Qnmpicpt}
Let $T\in \mathcal C(H)$ be densely defined, unbounded and $T^{\dagger}\in \mathcal K(H)$. Is it true that $T\in \mathcal{AM}_{c}(H)$?.
\end{question}
The Example \ref{counterexmpcompact} shows  that the answer to  Question \ref{Qnmpicpt} need not be affirmative. But we have an affirmative answer in a particular case.

\begin{proposition}\label{poscptgeninv}
Let $T\in \mathcal C(H)$ be positive, unbounded and $N(T)$ be finite-dimensional. If $T^{\dagger}\in \mathcal{K(H)}$, then $T\in \mathcal{AM}_{c}(H)$.
\end{proposition}
\begin{proof}
Since $N(T)$ reduces $T$, with respect to the subspaces $N(T)$ and $N(T)^{\bot}$, we can write $T$ by $T=0\oplus T_{0}$, where $T_0=T|_{N(T)^{\bot}}$. As $N(T)=N(Z_T)$, we have that $Z_T=0\oplus Z_{T_0}$. As $T_0^{-1}\in  \mathcal K(N(T)^{\bot})$, we can conclude that $Z_{T_0}\in \mathcal K(N(T)^{\bot})$. As $N(T)$ is finite-dimensional, by \cite[Proposition 3.6]{NBGR2}, it follows that $Z_{T}\in \mathcal{AM}(H)$. Now by Theorem \ref{equivalentAMs}, we can conclude that $T\in \mathcal{AM}_{c}(H)$.
\end{proof}
\begin{proposition}\label{cptgeninvisAM}
Let $T\in \mathcal C(H_1,H_2)$ be densely defined, unbounded and $N(T)$ be finite-dimensional. If  $T^{\dagger}\in \mathcal K(H_2,H_1)$, then $T \in \mathcal {AM}(H_1,H_2)$.
\end{proposition}
\begin{proof}
If $T^{\dagger}\in \mathcal K(H_2,H_1)$, then $|T^{{\dagger}^*}|\in \mathcal K(H_2)$. But by \cite[Proposition 3.19]{SHKGRAM1}, we have $|T|^{\dagger}=|T^{{\dagger}^*}|$. Since $N(T)=N(|T|)$,  by Proposition \ref{poscptgeninv}, we have $|T|\in \mathcal{AM}_{c}(H_1)$. Hence $T\in \mathcal{AM}_{c}(H_1,H_2)$.
\end{proof}

By combining Propositions \ref{cptmpiAM} and \ref{cptgeninvisAM}, we can state the following theorem.

\begin{theorem}\label{mpicharacterization}
Let $T\in \mathcal C(H_1,H_2)$ be densely defined and unbounded. Then $T\in \mathcal {AM}_{c}(H_1,H_2)$ if and only if $N(T)$ is finite-dimensional and $T^{\dagger}\in \mathcal K(H_2,H_1)$.
\end{theorem}

Next we ask whether the $\mathcal{AM}$-property of an operator imply the $\mathcal{AM}$-property of its adjoint. This question in the case of bounded operators is answered in \cite[Proposition 3.11]{NBGR2}. The unbounded case is discussed in \cite{SHKGRAM1} by assuming an extra condition, namely the invertibility of the operator.

\begin{theorem}\cite[Theorem 4.11]{SHKGRAM1}
If $T\in \mathcal C(H)$ is densely defined, unbounded and has a bounded inverse, then $T\in \mathcal{AM}_{c}(H)$ if and only if $T^*\in \mathcal {AM}_{c}(H)$.
\end{theorem}
Here we improve the above result by dropping the invertibility condition.
\begin{theorem}
Let $T\in \mathcal C(H_1,H_2)$ be densely defined. Assume that both $N(T)$ and $N(T^*)$ are finite-dimensional. Then $T\in \mathcal{AM}_{c}(H_1,H_2)$ if and only if $T^*\in \mathcal {AM}_{c}(H_2,H_1)$.
\end{theorem}
\begin{proof}
First assume that $T\in \mathcal{AM}_{c}(H_1,H_2)$. Then by Proposition \ref{cptmpiAM}, $T^{\dagger}\in \mathcal K(H_2,H_1)$. This implies that
$(T^*)^{\dagger}=(T^{\dagger})^*\in \mathcal{K}(H_1,H_2)$. As $N(T^*)$ is finite-dimensional, by Proposition \ref{poscptgeninv}, we can conclude that $T^*\in \mathcal{AM}(H_2,H_1)$. The converse follows easily by applying the above argument to $T^*$ and using the hypothesis that $N(T)$ is finite-dimensional, we can conclude that $T\in \mathcal{AM}_{c}(H_1,H_2)$.
  \end{proof}

%

Next we prove the spectral theorem for unbounded self-adjoint absolutely minimum attaining operators.
\begin{theorem}[Spectral theorem for self-adjoint $\mathcal{AM}$-operators]\label{selfadjequivalent}
  Let $H$ be an infinite-dimensional Hilbert space and $T\in \mathcal C(H)$ be self-adjoint. Assume that $T$ is unbounded. Then the following statements are equivalent.
  \begin{enumerate}
    \item\label{selfadjAM} $T\in \mathcal{AM}_{c}(H)$
    \item \label{selfadjMPinverse} $T^{\dagger} \in \mathcal K(H)$ and $N(T)$ is finite-dimensional
    \item \label{selfadjdiagonal} there exists a sequence $(\lambda_n)$ of real numbers and an orthonormal subset  ${\{v_n:n\in \mathbb N}\}$ of $H$ such that  $\displaystyle \lim_{n\rightarrow \infty}|\lambda_n|\rightarrow \infty$ and $Tv_n=\lambda_nv_n$ for each $n\in \mathbb N$. In this case, $D(T)=\left \{  x\in H: \displaystyle \sum_{n=1}^{\infty} \lambda_n^2|\langle x,v_n\rangle |^2<\infty \right \} $ and
 $$Tx=\displaystyle \sum_{n=1}^{\infty} \lambda_n \langle x,v_n\rangle v_n,\; \text{for all}\; x\in D(T)$$
    \item \label{selfadjpurespectrum} $T$ is diagonalizable, that is, $H$ has an orthonormal basis consisting of eigenvectors of $T$ and each eigenvalue of $T$ has finite multiplicity
    \item $\sigma(T) \subseteq \{0\} \cup \{\lambda_n: n \in \mathbb{N} \}$
and every spectral value is an eigenvalue with finite multiplicity, that is, $\sigma(T)=\sigma_d(T)$.
\item The resolvent $R_{\lambda}(T)=(T-\lambda I)^{-1}$ is compact for one, hence for all, $\lambda \in \rho(T )$.
    \item \label{selfadjimbedding} the embedding $J_{T}: (D(T),\|\cdot \|_{T})\rightarrow H$ is compact. Here $\|x\|_{T}=(\|x\|^2+\|Tx\|^2)^{\frac{1}{2}}$ for all $x\in D(T)$.
  \end{enumerate}
\end{theorem}
\begin{proof}
If $N(T)={\{0}\}$, then the result follows from \cite[Theorem 4.18]{SHKGRAM1}. Next assume that $N(T)\neq {\{0}\}$. Then the equivalence of (\ref{selfadjAM}) and (\ref{selfadjMPinverse}) follows by Propositions \ref{cptgeninvisAM} and \ref{cptmpiAM}. Since $N(T)$ is finite-dimensional and $N(T)$ reduces $T$, we have that $T=0\oplus T_0$, where $T_0=T_{N(T)^{\bot}}$, equivalence of the other statements follows by applying \cite[Theorem 4.18]{SHKGRAM1} to $T_0$ since $T_0$ is bijective and $T_0\in \mathcal{AM}_{c}(N(T)^{\bot})$.

To find $D(T)$ in (\ref{selfadjdiagonal}), we argue as follows; \\
$x\in D(T)$ if and only if $Tx\in H$ if and only if \begin{align*}
Tx&=\displaystyle \sum_{n=1}^{\infty}\langle Tx,v_n\rangle v_n\\
  &=\displaystyle \sum_{n=1}^{\infty}\langle x,Tv_n\rangle v_n\\
  &=\displaystyle \sum_{n=1}^{\infty}\lambda_n\langle x,v_n\rangle v_n.
  \end{align*}
  It is clear that $Tx\in H$ if and only if $\displaystyle \sum_{n=1}^{\infty}\lambda_n^2|\langle x,v_n\rangle|^2<\infty$. Hence $$D(T)=\left \{  x\in H: \displaystyle \sum_{n=1}^{\infty} \lambda_n^2|\langle x,v_n\rangle |^2<\infty \right \} $$ and $Tx=\displaystyle \sum_{n=1}^{\infty} \lambda_n \langle x,v_n\rangle v_n,\; \text{for all}\; x\in D(T).$
\end{proof}

It is natural to ask whether Theorem \ref{selfadjequivalent} holds true for normal operators. Here we answer this question.

\begin{theorem}[Spectral theorem for unbounded normal $\mathcal{AM}$-operators]\label{normaladjequivalent}
  Let $H$ be an infinite-dimensional Hilbert space and $T\in \mathcal C(H)$ be normal. Assume that $T$ is unbounded. Then the following statements are equivalent.
  \begin{enumerate}
    \item\label{normalAM} $T\in \mathcal{AM}_{c}(H)$
    \item \label{normalMPinverse} $T^{\dagger} \in \mathcal K(H)$ and $N(T)$ is finite-dimensional
    \item \label{normaldiagonal} there exists a sequence $(\lambda_n)$ of complex numbers and an orthonormal subset  ${\{v_n:n\in \mathbb N}\}$ of $H$ such that  $\lim_{n\rightarrow \infty}|\lambda_n|\rightarrow \infty$ and $Tv_n=\lambda_nv_n$ for each $n\in \mathbb N$. In this case, $D(T)=\left \{  x\in H: \displaystyle \sum_{n=1}^{\infty} \lambda_n^2|\langle x,v_n\rangle |^2<\infty \right \} $ and
 $$Tx=\displaystyle \sum_{n=1}^{\infty} \lambda_n \langle x,v_n\rangle v_n,\; \text{for all}\; x\in D(T)$$
    \item \label{normalpurespectrum} $T$ is diagonalizable, that is, $H$ has an orthonormal basis consisting of eigenvectors of $T$ and each eigenvalue of $T$ has finite multiplicity.
    \item \label{discretespectrum}$\sigma(T) \subseteq \{0\} \cup \{\lambda_n: n \in \mathbb{N} \}$ and every spectral value is an eigenvalue with finite multiplicity, that is, $\sigma(T)=\sigma_d(T)$.
    \item \label{cptresolvent} The resolvent $R_{\lambda}(T)=(T-\lambda I)^{-1}$ is compact for one, hence for all, $\lambda \in \rho(T)$.

    \item \label{normalimbedding} the embedding $J_{T}: (D(T),\|\cdot \|_{T})\rightarrow H$ is compact. Here $\|x\|_{T}=(\|x\|^2+\|Tx\|^2)^{\frac{1}{2}}$ for all $x\in D(T)$.
  \end{enumerate}
\end{theorem}
\begin{proof}
Equivalence of (\ref{normalAM}) and (\ref{normalMPinverse}) follows by Theorem \ref{mpicharacterization}. Next assume that (\ref{normalMPinverse}) is true. They by the spectral theorem for compact normal operators, there exists a sequence $(\mu_n)$ of complex numbers and an orthonormal system ${\{\phi_n}\}$ such that
\begin{equation*}
T^{\dagger}y=\sum_{n=1}^{\infty}\mu_n\langle y,\phi_n\rangle \phi_n,\; \text{for all}\; x\in H.
\end{equation*}
From the above equation it is clear that $T^{\dagger}(\phi_n)=\mu_n\phi_n$ for each $n\in \mathbb N$. That is, $\phi_n\in R(T^{\dagger})=C(T)$. It is also clear that $T^{\dagger}(y)\in C(T)$.
Since $T$ is normal, it is easy to observe that $N(T)=N(T^*)=N(T^{\dagger})$. Note that $\phi_n=P_{R(T)}(\phi_n)=TT^{\dagger}(\phi_n)$ for all $n\in \mathbb N$. Hence we have $T\phi_{n}=\lambda_n \phi_n$ for all $n\in \mathbb N$, where $$\lambda_n=\begin{cases}
0\quad  \mbox{if}\; \mu_n=0,\\
\mu_n^{-1} \quad  \mbox{if} \; \lambda_n\neq 0.
\end{cases}$$
Since $T^{\dagger}y\in C(T)$, $TT^{\dagger}y$ is meaningful.  Following the similar steps as in  (\ref{selfadjdiagonal}) of Theorem \ref{selfadjequivalent}, we can determine $D(T)$ easily. On the other hand, following similar arguments about $T$ and $T^{\dagger}$, we can easily prove the implication $(\ref{normaldiagonal}) \Rightarrow (\ref{normalMPinverse})$.

Proof of $(\ref{normaldiagonal})\Leftrightarrow (\ref{normalpurespectrum})$: By $(\ref{normaldiagonal})$ it is clear that $T|_{R(T)}$ is diagonalizable. Observe that if $0\in \sigma(T)$ and as it is an isolated point, it must be an eigenvalue of $T$ and hence by \cite[]{}, $R(T)$ must be closed. So we have $H=R(T)\oplus R(T)^{\bot}=R(T)\oplus N(T)$. As $N(T)$ is finite dimesnsional, we can take an orthonormal basis of it  and adjoin it with an orthonormal basis ${\{\phi_n:n\in \mathbb N}\}$ of $R(T)$ so that  we can obtain an  orthonormal basis of $H$ consisting of eigenvectors of $T$.

On the other hand, since $T$ is not bounded, we can conclude that $\sigma(T)$ is unbounded. It is easy to establish the representation of $T$ and determine the domain easily.

Proof of $(\ref{normalpurespectrum})\Leftrightarrow (\ref{discretespectrum})$: the implication $(\ref{normalpurespectrum})\Rightarrow (\ref{discretespectrum})$ is easy to prove. To prove the other way implication, let us assume that $\sigma(T)=\sigma_d(T)$. Let $\sigma(T)={\{\lambda_1,\lambda_2,\lambda_3,\dots}\}$. Write $G_n=N(T-\lambda_nI)$. It is easy to prove that each $G_n$ reduces $T$. By using the Zorn's Lemma, we can show that $H=\displaystyle \bigoplus_{n=1}^{\infty}G_n$ from which we can conclude that $T$ is diagonalizable.

Proof of $(\ref{discretespectrum})\Leftrightarrow (\ref{cptresolvent})$: Let us assume that $\sigma(T)=\sigma_{d}(T)={\{\lambda_1,\lambda_2,\lambda_3,\dots}\}$. Since $T$ is unbounded, $\sigma(T)$ is unbounded closed subset of $\mathbb C$. Hence $Tx=\displaystyle \sum_{n=1}^{\infty}\lambda_nP_nx$, where $P_n$ is an orthogonal projection onto $N(T-\lambda_nI)$ with $D(T)={\{x\in H:\displaystyle \sum_{n=1}^{\infty}|\lambda_n|^2\|P_nx\|^2<\infty}\}$. Hence  if $\lambda \in \rho(T)$, then
\begin{equation*}
(T-\lambda I)^{-1}=\displaystyle \sum_{n=1}^{\infty}(\lambda_n-\lambda)^{-1}P_n.
\end{equation*}
 It is easy to see that $|(\lambda_n-\lambda)^{-1}|\rightarrow 0$ as $n\rightarrow \infty$. From this it is easy to see that $(T-\lambda I)^{-1}\in \mathcal K(H)$.

Proof of $(\ref{cptresolvent})\Leftrightarrow (\ref{normalimbedding})$: the proof this follows as in the similar lines of the self-adjoint case. We refer to \cite[Theorem 5.12]{schmudgen} for details.
\end{proof}

Next, we turn our attention to invariant and hyperinvariant subspaces. Recall that a closed subspace $M$ of $H$ is said to be invariant under a densely defined closed operator $T$ if $T(D(T)\cap M)\subseteq M$, and hyperinvariant if $M$ is invariant under  every $A\in \mathcal B(H)$ such that
$AT\subseteq TA$.

If $T\in \mathcal{AM}_{c}(H)$ is unbounded and has a bounded inverse, then $T$ has a nontrivial hyperinvariant subspace. Here we remove the invertibility condition and prove the same result.

\begin{proposition}\label{hypinvsubsp}
  If $T\in \mathcal {AM}_{c}(H)$ is densely defined and unbounded, then $T$ has a non trivial hyperinvariant subspace.
\end{proposition}
\begin{proof}
If $T$ is bijective, then  $T^{-1}\in \mathcal B(H)$. Hence by \cite[Theorem 4.14]{SHKGRAM1}, $T$ has a hyperinvariant subspace. So assume that $T$ is not bijective. Note that $R(T)$ is closed by \cite[Proposition 4.2]{SHKGRAM1}.  If $T$ is one-to-one but not onto, then  $R(T)$ is a non trivial hyperinvariant subspace for $T$. If $T$ is onto but not one-to-one, then $N(T)$ is a hyperinvariant subspace for $T$.

\end{proof}

\begin{corollary}
Assume $T\in \mathcal C(H)$ is densely defined, unbounded. If  $R(T)$ is closed and $T^{\dagger}\in \mathcal K(H)$, then $T$ has a hyperinvariant subspace.
\end{corollary}
\begin{proof}
  First note that if $N(T)\neq {\{0}\}$, then clearly $N(T)$ is a hyperinvariant subspace for $T$. So assume that $N(T)={\{0}\}$. By Theorem \ref{cptgeninvisAM}, it follows that $T\in \mathcal{AM}_{c}(H)$. Hence by Theorem \ref{hypinvsubsp}, it is clear that $T$ has a hyperinvariant subspace.
\end{proof}

\subsection*{Acknowledgments}
The second author’s research is supported by SERB grant No. MTR/2019/001307, Govt. Of India.

\bibliographystyle{amsplain}

\begin{thebibliography}{10}

\bibitem{NBGR3}
N. Bala\ and\ G. Ramesh, A Bishop-Phelps-Bollob\'{a}s type property for minimum attaining operators, Oper. Matrices {\bf 15} (2021), no.~2, 497--513.
\bibitem{NBGR2}
N. Bala\ and\ R. Golla, Spectral properties of absolutely minimum attaining operators, Banach J. Math. Anal. {\bf 14} (2020), no.~3, 630--649.

\bibitem{NBGR4}
Neeru Bala and Golla Ramesh, Hyperinvariant subspace for absolutely norm attaining and absolutely minimum attaining operators, Preprint (2020), https://arxiv.org/abs/2002.09167

\bibitem{ben}
A. Ben-Israel\ and\ T. N. E. Greville, {\it Generalized inverses: theory and applications}, Wiley-Interscience, New York, 1974.


\bibitem{birmannsolomyak}
M. \v S. Birman\ and\ M. Z. Solomjak, {\it Spectral theory of selfadjoint operators in Hilbert space} (Russian), Leningrad. Univ., Leningrad, 1980.


\bibitem{carvajalneves1}
X. Carvajal\ and\ W. Neves, Operators that achieve the norm, Integral Equations Operator Theory {\bf 72} (2012), no.~2, 179--195.

\bibitem{carvajalneves2}
X. Carvajal\ and\ W. Neves, Operators that attain their minima, Bull. Braz. Math. Soc. (N.S.) {\bf 45} (2014), no.~2, 293--312.



\bibitem{GRS2}
J. Ganesh, G. Ramesh\ and\ D. Sukumar, A characterization of absolutely minimum attaining operators, J. Math. Anal. Appl. {\bf 468} (2018), no.~1, 567--583.

\bibitem{GRS3}
J. Ganesh, G. Ramesh\ and\ D. Sukumar, Perturbation of minimum attaining operators, Adv. Oper. Theory {\bf 3} (2018), no.~3, 473--490.

\bibitem{goldberg}
S. Goldberg, {\it Unbounded linear operators: Theory and applications}, McGraw-Hill, New York, 1966.

\bibitem{kato}
T. Kato, {\it Perturbation theory for linear operators}, second edition, Springer, Berlin, 1976.

\bibitem{SHKGRAM1}

S. H. Kulkarni\ and\ G. Ramesh, Absolutely minimum attaining closed operators, J. Anal. {\bf 29} (2021), no.~2, 473--492.

\bibitem{shkgrminattaining}
S. H. Kulkarni\ and\ G. Ramesh, On the denseness of minimum attaining operators, Oper. Matrices {\bf 12} (2018), no.~3, 699--709.
\bibitem{lance}
E. C. Lance, {\it Hilbert $C^*$-modules}, London Mathematical Society Lecture Note Series, 210, Cambridge University Press, Cambridge, 1995.

\bibitem{SP}
S. K. Pandey\ and\ V. I. Paulsen, A spectral characterization of $\mathcal{AN}$ operators, J. Aust. Math. Soc. {\bf 102} (2017), no.~3, 369--391.

\bibitem{radjavirosenthal}
H. Radjavi\ and\ P. Rosenthal, {\it Invariant subspaces}, second edition, Dover, Mineola, NY, 2003.


\bibitem{rameshthesis}
G. Ramesh, Approximation methods for solving operator equations involving unbounded operators, Ph.D. Thesis, IIT Madras. (2008).

\bibitem{RAMpara}
G. Ramesh, Absolutely norm attaining paranormal operators, J. Math. Anal. Appl. {\bf 465} (2018), no.~1, 547--556.
\bibitem{GROH1}
G. Ramesh\ and\ H. Osaka, On a subclass of norm attaining operators, Acta Sci. Math. (Szeged) {\bf 87} (2021), no.~1-2, 247--263.

\bibitem{GROH2}
G. Ramesh\ and\ H. Osaka, On  operators which their attain norm on every reducing subspace, to appear in Annals of Functional Analysis.
\bibitem{reedsimon1}
M. Reed\ and\ B. Simon, {\it Methods of modern mathematical physics. I}, second edition, Academic Press, Inc., New York, 1980.

\bibitem{rud}
W. Rudin, {\it Functional analysis}, second edition, International Series in Pure and Applied Mathematics, McGraw-Hill, New York, 1991.

\bibitem{schmudgen}
K. Schm\"{u}dgen, {\it Unbounded self-adjoint operators on Hilbert space}, Graduate Texts in Mathematics, 265, Springer, Dordrecht, 2012.



\bibitem{Sharifigap}
K. Sharifi, The gap between unbounded regular operators, J. Operator Theory {\bf 65} (2011), no.~2, 241--253.


\bibitem{taylorlay}
A. E. Taylor\ and\ D. C. Lay, {\it Introduction to functional analysis}, second edition, John Wiley \& Sons, New York, 1980.





%





















\bibitem{VNGR1}
D. Venku Naidu\ and\ G. Ramesh, On absolutely norm attaining operators, Proc. Indian Acad. Sci. Math. Sci. {\bf 129} (2019), no.~4, Paper No. 54, 17 pp.


\end{thebibliography}

\end{document}